\documentclass{amsart}
\usepackage[T2A]{fontenc}
\usepackage[cp1251]{inputenc}
\usepackage[english,russian]{babel}
\usepackage{amsfonts,amssymb,amscd,textcomp}
\usepackage{hyperref}
\usepackage{amsrefs}
\theoremstyle{plain}
\newtheorem{theorem}{Theorem}[section]
\newtheorem{lemma}{Lemma}[section]

\newtheorem{definition}{Definition}[section]

\newtheorem{remark}{Remark}[section]
\newtheorem{example}{Example}[section]
\numberwithin{equation}{section}

\begin{document}
\title[Pentagon Equation and Compact Quantum Semigroups]{Pentagon Equation and Compact Quantum Semigroups}
\author{Marat A. Aukhadiev}
\begin{abstract}{ Following the principles of  Kac algebras and compact quantum groups duality theories, we study the adoptability of multiplicative unitary theory for compact quantum semigroups. Particularly, the non-existence of such operator for several nontrivial examples of quantum semigroups is proved. Using the pentagon equation an alternative operator which encodes compact quantum semigroup structure is proposed, extending some aspects of multiplicative unitaries.}
\end{abstract}

\maketitle


\section{Introduction}
After the discover of Pontryagin duality for abelian groups, it was extended to non-abelian groups by Tannaki \cite{Tannaka}, Krein \cite{Krein}, and Stinespring, Tatsuuma. In 1963 a notion of ``ring group'' based on Stinespring duality theory was defined by Kac. Introducing this new object he proposed new approach to duality theory of unimodular locally compact groups. Kac defined notions of such homomorphisms as comultiplication, counit and antipode. The algebra with these homomorphisms was later called Kac algebra. Following this approach Takesaki \cite{Takesaki} defined a group algebra of locally compact group as an involutive commutative Hopf - von Neumann algebra with left-invariant measure. In this theory a crucial role is played by a unitary operator, later called Kac-Takesaki operator.

Let $G$ be a locally compact group with left Haar measure $ds$. Denote by $H$ Hilbert space $L^2(G,ds)$. Define a unitary operator $u$ on $H\otimes H$:
$$(uf)(s,t)=f(st,t),\ f\in H\otimes H,\ s,t\in G.$$
This operator, called  Kac-Takesaki operator, satisfies an original property -- \emph{pentagon equation}:
$$u_{12}u_{13}u_{23}=u_{23}u_{12}.$$
Let $\mathfrak{A}(G)$ be a Hopf - von Neumann algebra of all multiplicators by $\varrho(f)$, $f\in L^\infty(G)$ on $H$. The main fact about this group duality theory is the following. The map $\delta_G\colon x\to u(x\otimes 1)u^*$ is an isomorphism from $\mathfrak{A}(G)$ to $\mathfrak{A}(G)\bar{\otimes}\mathfrak{A}(G)$ such that $(\delta_G\otimes \mathrm{id})\delta_G=(\mathrm{id}\otimes\delta_G)\delta_G$. The pair $(\mathfrak{A}(G),\delta_G)$ is the main example of Hopf - von Neumann algebra. The homomorphism $\delta_G$ is called a comultiplication.

Soon after V.G. Drinfeld discovered quantum groups in 1980, which immediately became popular, there were attempts to generalize Pontryagin duality to this class of objects. S.L. Woronowicz \cite{Woronowicz} and also L.L.Vaksman and Y.S. Soibelman proposed a $C^*$-algebraic approach to quantum groups and constructed an example of quantum group $SU_q(2)$. Later for this approach Woronowicz defined a notion of compact quantum group. The duality in this new theory  is also based on the Kac-Takesaki operator. Baaj and Skandalis showed that the unitary operator satisfying the pentagon equation and some regularity properties gives a structure of compact quantum group \cite{Skandalis}. This operator was called a multiplicative unitary.

 In \cite{Vandaele} A.Van Daele defined a natural generalization of compact quantum groups, called compact quantum semigroups. The question of existence of multiplicative unitary for these objects arised soon after examples of quantum semigroups  appeared in literature. 
 In this paper the generalization of multiplicative unitary notion for compact quantum semigroups is investigated. It turns out that such operator $u$ may not exist for some compact quantum semigroups. This fact is proved in Section 3 for the Toeplitz algebra and the algebra of continuous functions on compact semigroup with zero. Despite this, in Section 4 by means of pentagon equation we define an operator on C*-algebra, which gives a non-trivial comultiplication, and thus it generalizes the multiplicative unitary. The suggested operator, dislike the multiplicative unitary, is suitable for the compact quantum semigroups with non-faithful Haar functional. Also we prove the existence of such operator for some special class of compact quantum semigroups.
 
 The author is thankful to A. Van Daele for useful comments.

\section{Preliminaries}

Let $\mathcal{A}$ be a $C^*$-algebra. Denote by $\mathcal{A}\otimes\mathcal{A}$ the minimal $C^*$-tensor product of $\mathcal{A}$ on itself. A unital $C^*$-homomorphism $\Delta\colon\mathcal{A}\to\mathcal{A}\otimes\mathcal{A}$ is called a \emph{comultiplication}, if the \emph{coassociativity} relation holds:
\begin{equation}\label{coass}
(\Delta\otimes id)\Delta=(id\otimes\Delta)\Delta.
\end{equation}

Then $(\mathcal{A},\Delta)$ is called a \emph{compact quantum semigroup} \cite{Vandaele}. If the subspaces
\begin{equation}\label{plotl}
\left\{\Delta(b)(a\otimes I); \ a,b\in\mathcal{A}\right\},
\end{equation}
\begin{equation}\label{plotr}
\left\{\Delta(b)(I\otimes a); \ a,b\in\mathcal{A}\right\},
\end{equation}
are dense in  $\mathcal{A}\otimes\mathcal{A}$, then $(\mathcal{A},\Delta)$ is a  \emph{compact quantum group} \cite{Woronowicz}.

\begin{definition}
A counit is a $*$-homomorphism $\epsilon\colon\mathcal{A}\to\mathbb{C}$, such that for each $a\in\mathcal{A}$
$$(\epsilon\otimes id)\Delta(a)=a,$$
$$(id\otimes\epsilon)\Delta(a)=a.$$
\end{definition}

Consider the dual space $\mathcal{A}^*$. The comultiplication $\Delta$ naturally generates multiplication $\ast$ in $\mathcal{A}^*$, for any $\rho,\varphi\in\mathcal{A}^*$ we have:
$$(\rho\ast\varphi)(a)=(\rho\otimes\varphi)\Delta(a).$$

State $h\in\mathcal{A}^*$ is a Haar functional in $\mathcal{A}^*$, if the following holds for any $\rho\in\mathcal{A}^*$:
\begin{equation}\label{haar}
h\ast{\rho}=\rho\ast{h}=\lambda_{\rho}h,
\end{equation}
where $\lambda_{\rho}\in\mathbb{C}$ depends on $\rho$.

Obviously, (\ref{haar}) is equivalent to
\begin{equation}\label{haarl}
(h\otimes id)\Delta(a)=h(a)I,
\end{equation}
\begin{equation}\label{haarr}
(id\otimes h)\Delta(a)=h(a)I.
\end{equation}

The following result is proved in \cite{Woronowicz}.

\begin{theorem}\label{wor1}
Every compact quantum group $(\mathcal{A},\Delta)$ admits unique Haar functional $h$. This functional is faithful on the dense subalgebra in $\mathcal{A}$.
\end{theorem}

The following notion of multiplicative unitary can be found in \cite{Skandalis}.

Let $H$ be a separable Hilbert space. Denote the flip by $\sigma\colon H\otimes H\to H\otimes H$:
\begin{equation}\label{flip}
\sigma(x\otimes y)=(y\otimes x), \ x,y\in H.
\end{equation}

For any $a\in B(H\otimes H)$ introduce the following notation:
\begin{equation}\label{123}
a_{12}=a\otimes I, a_{23}=I\otimes a, a_{13}=\sigma_{12}a_{23}\sigma_{12}=\sigma_{23}a_{12}\sigma_{23}.
\end{equation}

\begin{definition}
Operator $u\in B(H\otimes H)$ is a multiplicative isometry if it is an isometric operator which satisfies the pentagon equation:
\begin{equation}\label{pentagon}
u_{12}u_{13}u_{23}=u_{23}u_{12}.
\end{equation}
\end{definition}
If such $u$ is unitary, it is called a \emph{multiplicative unitary}.

Consider the GNS-construction $(H_{\varphi},\pi_{\varphi})$ of $\mathcal{A}$, associated with a state $\varphi$. Denote by $N_{\varphi}$ the left ideal $\left\{a\in\mathcal{A}| \ \varphi(a^*a)=0\right\}$. For any element $a\in\mathcal{A}$ denote by $\overline{a}$ the corresponding equivalence class in space $\mathcal{A}/N_{\varphi}$.

Baaj and Skandalis achieved the following result for compact quantum groups in \cite{Skandalis}.

\begin{theorem}\label{th3}
Let $(\mathcal{A},\Delta)$ be a compact quantum group with Haar functional $h$ and $(H_{h},\pi_{h})$ be the GNS-representation associated with $h$. Then there exists a multiplicative unitary $u\in B(H_h\otimes H_h)$, such that
\begin{equation}
(\pi_{h}\otimes\pi_{h})\Delta(a)=u(\pi_{h}(a)\otimes I)u^*.\label{eq2.12}
\end{equation}
\end{theorem}

\section{Multiplicative unitaries and compact quantum semigroups}

It is known from \cite{Woronowicz} that  the Haar functional exists for every compact quantum group. But since the proof of this statement is based on the density conditions (\ref{plotl}),(\ref{plotr}), this may not be the case for all compact quantum semigroups. However, (\ref{haar}) implies uniqueness of Haar functional if it exists. 

The following shows existence of multiplicative isometry in the case of compact quantum semigroup with Haar functional.

\begin{theorem}\label{th2}
Let $(\mathcal{A},\Delta)$ be a compact quantum semigroup with Haar functional $h$ and $(H_{h},\pi_{h})$ -- the GNS-representation associated with $h$. Then there exists a multiplicative isometry $u\in B(H_h\otimes H_h)$, such that 
\begin{equation}
u^*(\pi_{h}\otimes\pi_{h})(\Delta(a))u=\pi_{h}(a)\otimes I.\label{eq2.10}
\end{equation}
\end{theorem}
\begin{proof}
For any $a,b\in\mathcal{A}$ define
\begin{equation}
u(\overline{a}\otimes\overline{b})=\overline{\Delta(a)(I\otimes b)}.\label{eq2.11}
\end{equation}
Since $h$ is a Haar functional, using (\ref{haarl}) we obtain:
$$\left\langle u(\overline{a}\otimes\overline{b}),u(\overline{a}\otimes\overline{b})\right\rangle=(h\otimes h)((I\otimes b^*)\Delta(a^*)\Delta(a)(I\otimes b))=$$
$$=h((h\otimes id)((I\otimes b^*)\Delta(a^*a)(I\otimes b)))=h(b^*(h\otimes id)\Delta(a^*a)b)=$$
$$=h(b^*b)h(a^*a)=\left\langle \overline{a}\otimes\overline{b},\overline{a}\otimes\overline{b}\right\rangle$$
Consequently $u$ is isometry and the following verifies equality (\ref{pentagon})
$$u_{12}u_{13}u_{23}(\overline{a}\otimes\overline{b}\otimes\overline{c})=u_{12}u_{13}(\overline{a}\otimes u(\overline{b}\otimes \overline{c}))=u_{12}u_{13}(\overline{a}\otimes \overline{\Delta(b)(I\otimes c)})=$$
$$=u_{12}u_{13}\overline{((a\otimes\Delta(b))(I\otimes I\otimes c))}=\overline{(\Delta\otimes id)(\Delta(a))(id\otimes\Delta)(I\otimes b)(I\otimes I\otimes c)}=$$
$$=\overline{(id\otimes\Delta)(\Delta(a)(I\otimes b))(I\otimes I\otimes c)}=u_{23}u_{12}(\overline{a}\otimes\overline{b}\otimes\overline{c}).$$

Thus, $u$ is a multiplicative isometry. One can easily check relation (\ref{eq2.10}).
\end{proof}

The next example shows that the operator $u$ in theorem~\ref{th3} is not unitary for compact quantum semigroups in general.    Nevertheless, by theorem~\ref{th2} the multiplicative isometry exists, only the existence of Haar functional is required.

\begin{example}\label{pr1}
\end{example}

Consider the Toeplitz algebra $\mathcal{T}$ --- a minimal $C^*$-algebra, generated by an isometric unilateral shift operator $T$ and $T^*$ on a Hilbert space $H$ with basis $\left\{e_n\right\}^{\infty}_{n=0}$. Linear combinations of operators $T^nT^{*m}$ for positive integers $n,m$ are dense in the algebra $\mathcal{T}$. Such operators are denoted by $T_{n,m}$. It was proved in \cite{Aukh} that this algebra admits comultiplication $\Delta$ and the Haar functional $h$ given by:
\begin{equation}
\Delta(T)=T\otimes T,
\end{equation}
\begin{equation}\label{eq2.14}
h(I)=1,~h(T_{n,m})=0,
\end{equation}
for all $m,n\in\mathbb{N}$ except for $m=n=0$. This shows that the Toeplitz algebra can be regarded as an algebra of functions on a quantum semigroup with Haar measure.

 Obviously, $h$ is not faithful even on the dense subalgebra in $\mathcal{T}$.

\begin{lemma}\label{lem1}
The GNS representation of $\mathcal{T}$ associated with $h$ gives the same Toeplitz algebragenerated by unilateral shift operator $\pi_{h}(T)$ on $H_{h}$.

\begin{proof}
One can easiy check that $N_h$ is generated by the family of operators $\left\{T_{n,m}\right\}_{m\neq0}$. Let $H$ be a Hilbert space with basis $\left\{e_k\right\}^{\infty}_{k=0}$ where $e_k=\left[T^k\right]=T^k+N_h$. Then the corresponding GNS representation $\pi_{h}$ acts in the following way:
$$\pi_{h}(T^n)e_k=\pi_{h}(T^n)\left[T^k\right]=\left[T^{n+k}\right]=e_{n+k}.$$

Thus, $\pi_{h}(\mathcal{T})$ is isomorphic to $\mathcal{T}$, and $\pi_h(T)$ is a unilateral shift.

\end{proof}

\end{lemma}

\begin{theorem}\label{th4}
There does not exist the multiplicative unitary for $(\mathcal{T},\Delta)$ which satisfies (\ref{eq2.12}).
\end{theorem}
\begin{proof}
Suppose $u\in B(H\otimes H)$ is a unitary which satisfies (\ref{eq2.12}). Particularly, the following relations hold:
\begin{equation}
\label{ut} u(T\otimes I)=(T\otimes T)u,
\end{equation}

\begin{equation}
\label{uts}u(T^*\otimes I)=(T^*\otimes T^*)u,
\end{equation}
\begin{equation}
\label{us} u^*(T^*\otimes T^*)=(T^*\otimes I)u^*.
\end{equation}
 The relation (\ref{us}) implies that $H$ contains an orthonormal system $\{x_i,\ i=0,1,2... \}$ such that $u^*(e_0\otimes e_i)=e_0\otimes x_i$.
 Suppose $u$ satisfies pentagon equation. Then $u^*$ satisfies the folowing:
 \begin{equation}
 \label{ps} u_{23}^*u_{13}^*u_{12}^*=u_{12}^*u_{23}^*.
 \end{equation}
Apply the left-hand side of this equation to $(e_0\otimes e_0\otimes e_i)$, for $i \geq 0$:
$$u_{23}^*u_{13}^*u_{12}^*(e_0\otimes e_0\otimes e_i)=u_{23}^*u_{13}^*(e_0\otimes x_0\otimes e_i)=u_{23}^*(e_0\otimes x_0\otimes x_i) =$$ $$= e_0\otimes u^*(x_0\otimes x_i).$$
Calculate the right-hand side of (\ref{ps}) on the same vector:
$$u_{12}^*u_{23}^*(e_0\otimes e_0\otimes e_i)=u_{12}^*(e_0\otimes e_0\otimes x_i)=e_0\otimes x_0\otimes x_i.$$
 Consequently, \begin{equation}\label{ux} u^*(x_0\otimes x_i)=x_0\otimes x_i \mbox{ for any } i\geq 0.\end{equation} By virtue of (\ref{ut}), we have:
 $$u(e_j\otimes x_i)=e_j\otimes e_{j+i},\ \mbox{ for any } i,j\geq 0.$$
 
 Consider $x_0=\sum\limits_{i=0}^{\infty}\alpha_i e_i$ and apply the left-hand side of (\ref{uts}) to $x_0\otimes x_0$:
 $$ u(T^*\otimes I)(x_0\otimes x_0)=u( \sum\limits_{j=0}^{\infty}\alpha_{j+1}e_j\otimes x_0)=$$
 
 $$=\sum\limits_{j=0}^{\infty}\alpha_{j+1} u(e_j\otimes x_0)=\sum\limits_{j=0}^{\infty}\alpha_{j+1} e_j\otimes e_{j}.$$
 
 Calculate the right-hand side of (\ref{uts}) on $x_0\otimes x_0$, using (\ref{ux}).
 $$(T^*\otimes T^*)u(x_0\otimes x_0)= (T^*\otimes T^*)(x_0\otimes x_0)=(T^*\otimes T^*)(\sum\limits_{i=0}^{\infty}\alpha_i e_i\otimes \sum\limits_{j=0}^{\infty}\alpha_j e_j)=$$
 $$=\sum\limits_{i,j=0}^{\infty}\alpha_{i+1}\alpha_{j+1} e_i\otimes e_j. $$
 It implies that $ \alpha_i^2=\alpha_i$, $\alpha_i\alpha_j=0$ for any $i,j\geq 1$ such that $i\neq j$. Since $\|x_0\|=1$, $x_0=e_0$ or $x_0=e_n$ for some $n\geq 1$.
 
 Suppose $x_0=e_n$ for $n\geq 1$. By virtue of (\ref{ux}) and applying (\ref{ut}) to $e_0\otimes x_i$ we obtain:
 $$x_i=e_{n+i} \mbox{ for any } i\geq 0.$$
 Consequently, $u(e_0\otimes e_{n+i})=e_0\otimes e_i$, $u^*(e_0\otimes e_i)=e_0\otimes e_{n+i}$ for any $i\geq 0$. Вy (\ref{us}), the space $H$ contains an orthonormal system $\{y_j,\ j=1,2,... \}$ such that $u^*(e_j\otimes e_0)=e_0\otimes y_j$. Since $e_j\otimes e_0 \perp e_0\otimes e_i$ for any $i,j\geq 1$, the linear span of $e_0,e_1,...,e_{n-1}$ contains all $y_j$.   So, in this case $u$ is not unitary.
 
 Suppose $x_0=e_0$, then $u(e_0\otimes e_0)=e_0\otimes e_0,$ $u(e_0\otimes e_i)=e_0\otimes e_i$ for every $i\geq 0$. Using (\ref{ut}) for all $i,k\geq 0$ we obtain:
 $$u(e_k\otimes e_i)=u(T^k\otimes I)(e_0\otimes e_i)=(T^k\otimes T^k)(e_0\otimes e_i)=e_k\otimes e_{k+i}.$$
 Consequently, the image of $H\otimes H$ under $u$ lies in linear span of $\{e_i\otimes e_j| j\geq i \}$. Thus $u$ is not unitary.
\end{proof}

This example shows that there exist compact quantum semigroups which do not admit multiplicative unitary. The following shows that  multiplicative isometry cannot replace multiplicative unitary.  

\begin{example}\label{pr2}
\end{example}

For some compact quantum semigroups the GNS representation associated with Haar functional and the multiplicative isometry are trivial, so they do not encode the formation about the basic object. Consider the algebra $C(\mathcal{S})$ of all continuous functions on a compact semigroup $\mathcal{S}$ with a zero element. Define comultiplication:

\begin{equation}
(\Delta(f))(x,y)=f(xy).
\end{equation}

By associativity of multiplication in $\mathcal{S}$, $(C(\mathcal{S}),\Delta)$ is a compact quantum semigroup. The Haar functional  $h$ on $(C(\mathcal{S}),\Delta)$ is given by:
$$h(f)=f(0).$$

Obviously, $h$ is not faithful, and the corresponding GNS representation is one-dimensional. Thus, the multiplicative unitary and the multiplicative isometry do not play the same role as in the case of compact quantum groups. 

\section{The Comultiplication}

The notion of multiplicative unitary is based on the GNS representation associated with the Haar functional $h$. Since $\pi_{h}$ is a faithful representation of the C*-algebra corresponding to the compact quantum group, $\Delta$ generates comultiplication $\Delta^{'}$ on $\pi_{h}(\mathcal{A})$, given by the following commutative diagram:

$$
\begin{CD}
\mathcal{A} @>\Delta>> \mathcal{A}\otimes\mathcal{A}\\
@V{\pi_{h}}VV  @VV{\pi_{h}\otimes{\pi_{h}}}V\\
\pi_{h}(\mathcal{A}) @>\Delta^{'}>> \pi_{h}(\mathcal{A})\otimes\pi_{h}(\mathcal{A})
\end{CD}
$$

Then the multiplicative unitary in Theorem~\ref{th3} satisfies the following condition:
\begin{equation}\label{eq3.1}
\Delta^{'}(a)=u(a\otimes 1)u^*,
\end{equation}
for any $a\in\pi_{h}(\mathcal{A})$.

The right-hand side of (\ref{eq3.1}) is the image of the following operator $W:\pi_{h}(\mathcal{A})\otimes\pi_{h}(\mathcal{A})\rightarrow\pi_{h}(\mathcal{A})\otimes\pi_{h}(\mathcal{A})$ on $a\otimes I$, where $W(\cdot)=u\cdot u^*$. The pentagon equation for $u$ is encoded in the same equation for $W$. In the case of a compact quantum group we can identify $\mathcal{A}$ and $\pi_{h}(\mathcal{A})$. Then $W$ is a linear invertible operator on $\mathcal{A}\otimes\mathcal{A}$, which encodes the comultiplication of $(\mathcal{A},\Delta)$:

\begin{equation}\label{eq3.2}
\Delta(a)=W(a\otimes 1)
\end{equation}

As it was mentioned in the previous section, operator $W$ cannot be obtained by the same method. Nevertheless, thee may still exist an operator with the same conditions as $W$ except the invertibility. The following explains the idea of this operator.

Let $\mathfrak{L}(\mathcal A)$ be the algebra of all linear continuous operators on the $C^*$-algebra $\mathcal A$. Denote by $\Sigma\in\mathfrak {L}(\mathcal A\otimes \mathcal A)$ the flip operator $\Sigma(a\otimes b)=b\otimes a$ for $a,b\in \mathcal A$.

For $V\in\mathfrak {L}(\mathcal A\otimes \mathcal A)$, $a\in \mathcal A\otimes \mathcal A$ denote \\
\hspace*{\fill}$V_{12}=V\otimes id, \qquad V_{23}=id\otimes V, \qquad  V_{13}=\Sigma_{12}V_{23}\Sigma_{12}=\Sigma_{23}V_{12}\Sigma_{23}$\hspace*{\fill}\\
\hspace*{\fill}$a_{12}=a\otimes 1, \qquad a_{23}=1\otimes a, \qquad a_{13}=(\Sigma\otimes id)(a_{23})=\Sigma_{12}a_{23}$\hspace*{\fill}\\

\begin{definition}
Let $\mathcal A$ be a $C^*$-algebra. We say that $W:\mathcal A\otimes \mathcal A\rightarrow \mathcal A\otimes \mathcal A$  satisfies pentagon equation if the following holds:
\begin{equation}
W_{12}W_{13}W_{23}=W_{23}W_{12}\label{eq3.3}
\end{equation}
\end{definition} 

The following statement is obvious.

\begin{lemma}
$\Delta^L:\mathcal A\rightarrow \mathcal A\otimes \mathcal A$, $a\mapsto a\otimes 1$ and $\Delta^R:\mathcal A\rightarrow \mathcal A\otimes \mathcal A$, $a\mapsto 1\otimes a$ define comultiplication on  $\mathcal A$. These comultiplications are called trivial.
\end{lemma}

Operator $W$ with several additional conditions endows $\mathcal A$ with comultiplication.

\begin{theorem}
For a unital $C^*$-algebra $\mathcal A$ and a unital $C^*$-homomorphism $W\in \mathfrak {L}(\mathcal A\otimes \mathcal A)$, which satisfies the pentagon equation, define operators $\Delta,\widehat{\Delta}:\mathcal A\to \mathcal A\otimes \mathcal A$:

$$\Delta=W\Delta^L, \  \widehat{\Delta}=\Sigma W\Sigma\Delta^R.$$

Then $(\mathcal A,\Delta)$ and $(\mathcal A,\widehat{\Delta})$ are compact quantum semigroups.
\end{theorem}
\begin{proof}

It is sufficient to check (\ref{coass}) for the maps $\Delta,\widehat{\Delta}$. For any $a\in \mathcal A\otimes \mathcal A$ we have:\\
\hspace*{\fill}
$(\Delta\otimes id)a=(W\Delta^L\otimes id)a=$
\hspace*{\fill}\\
\hspace*{\fill}
$=(W\otimes id)(\Delta^L\otimes id)a=(W\otimes id)a_{13}=W_{12}(a_{13})$
\hspace*{\fill}\\

Using this equation and relation (\ref{eq3.3}), we get:\\
\hspace*{\fill}
$(\Delta\otimes id)\Delta(a)=W_{12}(\Delta(a))_{13}=W_{12}(W(a\otimes 1))_{13}=W_{12}W_{13}(a\otimes 1)_{13}=$
\hspace*{\fill}\\
\hspace*{\fill}
$=W_{12}W_{13}(a\otimes 1\otimes 1)=W_{12}W_{13}(id\otimes W)(a\otimes 1\otimes 1)=$
\hspace*{\fill}\\
\hspace*{\fill}
$=W_{12}W_{13}W_{23}(a\otimes 1\otimes 1)=W_{23}W_{12}(a\otimes 1\otimes 1)=W_{23}(W(a\otimes 1))_{12}=$
\hspace*{\fill}\\
\hspace*{\fill}
$W_{23}(\Delta(a))_{12}=(id\otimes \Delta)\Delta(a).$
\hspace*{\fill}\\
Analogously we get coassociativity of $\widehat{\Delta}$.
\end{proof}

So,  homomorphism $W$  on $\mathcal {A}\otimes\mathcal{ A}$ satisfying  the pentagon equation endows $\mathcal{A}$ with the compact quantum semigroup structure. The following shows that  a compact quantum semigroup with some additional conditions admits such an operator $W$.

\begin{theorem}\label{th3.2}
Let $(\mathcal A,\Delta)$ be an arbitrary compact quantum semigroup witha couint $\epsilon$. Then there exist  $C^*$-homomorphisms $W^L,W^R\in\mathfrak {L}(\mathcal A\otimes \mathcal A)$, which satisfy the pentagon equation, such that:

\begin{equation}\label{eq3.4}
\Delta=W^L\Delta^L=W^R\Delta^R.
\end{equation}
\end{theorem}
\begin{proof}
Consider operators $W^L=\Delta(id\otimes \epsilon)$, $W^R=\Delta(\epsilon\otimes id)$. The counit conditions imply relation (\ref{eq3.4}).
\end{proof}

\begin{remark}
Operators $W^L, W^R$ in theorem~\ref{th3.2} are projections and satisf the following conditions:\\
\hspace*{\fill}
$W^LW^R=W^R, W^RW^L=W^L.$
\hspace*{\fill}\\
Moreover, there exist $C^*$-homomorphisms $W^{L^{'}},W^{R^{'}}\in\mathfrak {L}(\mathcal A\otimes \mathcal A)$,\\
\hspace*{\fill}
$W^{L^{'}}=(id\otimes \epsilon(\cdot)I), W^{R^{'}}=(\epsilon(\cdot)I\otimes id),$
\hspace*{\fill}\\
which are projections and the followig equations hold:
$$W^{L^{'}}\Delta =\Delta^L, W^{R^{'}}\Delta =\Delta^R,$$
$$W^LW^{L^{'}}=W^L, W^RW^{R^{'}}=W^R.$$\end{remark}

Let us show that the examles of compact quantum semigroups from the previous section are given by such operator satisfying the pentagon equation.

For a compact quantum semigroup in~\ref{pr1} the existence of the counit is shown in \cite{Aukh}, given by

$$\epsilon(T_{n,m})=1,$$
for all $n,m\in\mathbb{Z}_+.$
By virtue of theorem~\ref{th3.2}, there exist  homomorphisms  $W^L$ and $W^R$ on $\mathcal{T}\otimes\mathcal{T}$ satisfying relations  (\ref{eq3.4}), given by:
$$W^L(T_{n,m}\otimes T_{k,l})=T_{n,m}\otimes T_{n,m},$$
$$W^R(T_{n,m}\otimes T_{k,l})=T_{k,l}\otimes T_{k,l}.$$

Consider the algebra of continuous functions $C(\mathcal{S}\times\mathcal{S})$ for a compact semigroup $\mathcal{S}$ and the following operator $W$ on it:
$$W(f)(x,y)=f(xy,y),$$
for $f\in C(\mathcal{S}\times\mathcal{S})$, $x,y\in\mathcal{S}$.
Then using $W$ the comultiplication $\Delta$ in example~\ref{pr2} can be constructed from $\Delta^L$:
$$\Delta=W\Delta^L.$$

Thus, these examples show that despite the fact that compact quantum semigroups do not fit  the multiplicative unitary theory in general, it is possible to define another operator which gives a structure of a compact quantum semigroup on an arbitrary C*-algebra. And this operator generalizes the multiplicative unitary.


\begin{bibdiv}\begin{biblist}


\bib{Aukh}{article}{
author={ M. A. Aukhadiev, S. A. Grigoryan, E. V. Lipacheva}
title={ Infinite-dimensional compact quantum semigroup},
journal={ Lobachevskii Journal of Mathematics},
date={ 2011},
volume= {32,4},
pages={ 304--316}}

\bib{Vandaele}{article}{author={A. Maes, A. Van Daele},
title={Notes on Compact Quantum Groups},
journal={Nieuw Arch. Wisk.},
date={1998},
volume={4,16},pages={73--112}}
\bib{Woronowicz}{article}{author={S.L. Woronowicz},
title={Compact quantum groups},
journal={Sym\'etries quantiques},
date={1998},
volume={ },pages={845--884}}

\bib{Krein}{article}{author={M.G.Krein},
title={A principle of duality for bicompact groups and quadratic block algebras},
journal={Doklady AN SSSR},
date={1949},
volume={69 },pages={725--728}}
\bib{Skandalis}{article}{author={Baaj S., Skandalis G.},
title={Unitaires multiplicatifs et dualit\'e pour les produits crois\'es de C*-alg\'ebres},
journal={Ann. Scient. Ec. Norm. Sup.},
date={1993},
volume={4,26 },pages={425--488}}
\bib{Takesaki}{article}{author={M.Takesaki},
title={Duality and von Neumann algebras},
journal={Bull. AMS},
date={1971},
volume={4,77 },pages={553--557}}
\bib{Tannaka}{article}{author={T.Tannaka},
title={Uber den Dualitatssatz der nichtkommutativen topologischen Gruppen},
journal={Tohoku Math.J.},
date={1939},
volume={45 },pages={1--12}}

\end{biblist}

\end{bibdiv}
\end{document}